\documentclass[12pt, a4paper]{article}
\usepackage[top=2.54cm, bottom= 2.54cm, left=2.54cm, right=2.54cm]{geometry}
\def\Dj{\hbox{D\kern-.73em\raise.30ex\hbox{-}
\raise-.30ex\hbox{}}}
\def\dj{\hbox{d\kern-.33em\raise.80ex\hbox{-}
\raise-.80ex\hbox{\kern-.40em}}}
\usepackage{epsfig}
\usepackage{amsmath,amsthm,amsfonts,amssymb,amscd,cite}
\usepackage{pstricks}

\allowdisplaybreaks

\newtheorem{thm}{Theorem}
\newtheorem{cor}[thm]{Corollary}

\newtheorem{lem}[thm]{Lemma}
\newtheorem{obs}[thm]{Observation}

\title{On Graphs with equal Dominating and C-dominating energy}

\author{\bf S. M. Hosamani$^{a}$, V. B. Awati$^{b}$ and R. M. Honmore$^{c}$\\[5mm]
{\it \normalsize $^{a,b,c}$ Department of Mathematics, Rani Channamma University,} \\[-1mm]
{\it \normalsize Belagavi, India} \\[2mm]
\normalsize e-mail$^{a}$: {\tt sunilkumar.rcu@gmail.com } \\[2mm]
\normalsize e-mail$^{b}$: {\tt awati\_vb@yahoo.com}}

\date{}

\date{}

\begin{document}

\maketitle

\begin{abstract}
Graph energy and Domination in graphs are most studied areas of graph theory. In this paper we made an attempt to connect these two areas of graph theory by introducing c-dominating energy of a graph $G$. First, we show the chemical applications of c-dominating energy with the help of well known statistical tools.   Next, we obtain mathematical properties of c-dominating energy. Finally, we characterize trees, unicyclic graphs, cubic and block graphs  with equal dominating and c-dominating energy.
\\[3mm]
\\[3mm]
{\bf \AmS \; Subject Classification:} 05C69; 05C90; 05C35; 05C12.
\end{abstract}

\baselineskip=0.30in

\section{Introduction}
All graphs considered in this paper are finite, simple and undirected. In particular, these graphs do not possess loops. Let $G=(V, E)$ be a graph with the vertex set $V(G)= \{v_{1}, v_{2}, v_{3}, \cdots, v_{n}\}$ and the edge set $E(G)= \{e_{1}, e_{2}, e_{3}, \cdots, e_{m}\}$, that is $|V(G)|= n$ and $|E(G)| = m$. The vertex $u$ and $v$ are adjacent if $uv\in E(G)$. The open(closed) neighborhood of a vertex $v \in V(G)$ is $N(v) = \{u: uv\in E(G)\}$ and $N[v] = N(v) \cup \{v\}$ respectively. The degree of a vertex $v \in V(G)$ is denoted by $d_{G}(v)$ and is defined as $d_{G}(v) = |N(v)|$. A vertex $v \in V(G)$ is pendant if $|N(v)| =1$ and is called support vertex if it is adjacent to pendant vertex. Any vertex $v \in V(G)$ with $|N(v)|>1$ is called internal vertex. If $d_{G}(v) =r$ for every vertex $v \in V(G)$, where $r \in \mathbb{Z}^{+}$ then $G$ is called r-regular. If $r=2$ then it is called cycle graph $C_{n}$ and for $r = 3$ it is called the cubic graph. A graph $G$ is unicyclic If $|V| = |E|$. A graph $G$ is called a block graph, if every block in $G$ is a complete graph. For undefined terminologies we refer the reader to \cite{harary}.
\vspace{2mm}

A subset $D\subseteq V(G)$ is called dominating set if $N[D]= V(G)$. The minimum cardinality of such a set $D$ is called the domination number $\gamma(G)$ of $G$. A dominating set $D$ is connected if the subgraph induced by $D$ is connected. The minimum cardinality of connected dominating set $D$ is called the connected dominating number $\gamma_{c}(G)$ of $G$ \cite{sampathkumar}.

The energy $E(G)$ of a graph $G$ is equal to the sum of the absolute values of the eigenvalues of the adjacency matrix of $G$. This quantity, introduced almost 30 years ago \cite{gutman1} and having a clear connection to chemical problems \cite{gutman3}, has in newer times attracted much attention of mathematicians and mathematical chemists \cite{coulson,consonni,cvetkovic,diaz,liu,milovanovc,rada,shparlinski,zhou}.

\indent In connection with energy (that is defined in terms of the eigenvalues of the adjacency matrix), energy-like quantities were considered also for the other matrices: Laplacian \cite{gutman3}, distance \cite{indulal}, incidence \cite{jooyandeh}, minimum covering energy \cite{adiga} etc. Recall that a great variety of matrices has so far been associated with graphs \cite{balban,bapat,cvetkovic,trinajstic}.\\

Recently in \cite{kanna} the authors have studied the dominating matrix which is defined as :\\
Let $G= (V, E)$ be a graph with $V(G) = \{v_{1}, v_{2}, \cdots, v_{n}\}$ and let $D \subseteq V(G)$ be a minimum dominating set of $G$.
The minimum dominating matrix of $G$ is the $n \times n$ matrix defined by $A_{D}(G) = (a_{ij})$, where $a_{ij} = 1$ if  $v_{i}v_{j} \in E(G)$ or $v_{i}=v_{j} \in D$, and $a_{ij} = 0$ if  $v_{i}v_{j} \notin E(G)$.  \\

The characteristic polynomial of $A_{D}(G)$ is denoted by $f_n(G, \mu) := det(\mu I-A_{D}(G))$.

The minimum  dominating eigenvalues of a graph $G$ are the eigenvalues of $A_{D}(G)$. Since $A_{D}(G)$ is real and symmetric, its eigenvalues are real numbers and we label them in non-increasing order $\mu_1 \ge \mu_2 \ge \cdots \ge \mu_n$. The minimum  dominating energy of $G$ is then defined as

\begin{center}
$E_{D}(G) = \sum \limits _{i=1}^n |\mu_i|$.
\end{center}

Motivated by dominating matrix, here we define the minimum connected dominating matrix abbreviated as (c-dominating matrix). The c-dominating matrix of $G$ is the $n \times n$ matrix defined by $A_{D_c}(G) = (a_{ij})$, where
\begin{center}
    $a_{ij} = \left\{
       \begin{array}{ll}
         1, & \hbox{if $v_{i}v_{j} \in E$;}
				\\[2mm]
         1, & \hbox{if $i = j$ and $v_{i} \in D_{c}$;}
				\\[2mm]
         0, & \hbox{otherwise.}
       \end{array}
     \right.
    $
\end{center}

The characteristic polynomial of $A_{D_c}(G)$ is denoted by $f_n(G, \lambda) := det(\lambda I-A_{D_c}(G))$.

The c-dominating eigenvalues of a graph $G$ are the eigenvalues of $A_{D_c}(G)$. Since $A_{D_c}(G)$ is real and symmetric, its eigenvalues are real numbers and we label them in non-increasing order $\lambda_1 \ge \lambda_2 \ge \cdots \ge \lambda_n$. The c-dominating energy of $G$ is then defined as

\begin{center}
$E_{D_c}(G) = \sum \limits _{i=1}^n |\lambda_i|$.
\end{center}
To illustrate this, consider the following examples:

\begin{center}
% Generated with LaTeXDraw 2.0.2
% Sun Sep 28 22:23:17 IST 2014
% \usepackage[usenames,dvipsnames]{pstricks}
% \usepackage{epsfig}
% \usepackage{pst-grad} % For gradients
% \usepackage{pst-plot} % For axes
%\scalebox{1} % Change this value to rescale the drawing.
{
\begin{pspicture}(0,-0.33890626)(5.6228123,0.30890626)
\psdots[dotsize=0.2](1.6809375,0.17890625)
\psdots[dotsize=0.2](2.4809375,0.17890625)
\psdots[dotsize=0.2](3.2409375,0.17890625)
\psdots[dotsize=0.2](3.9609375,0.19890624)
\psdots[dotsize=0.2](4.7009373,0.19890624)
\psline[linewidth=0.02cm](1.7209375,0.19890624)(2.5009375,0.19890624)
\psline[linewidth=0.02cm](2.5009375,0.19890624)(3.2609375,0.19890624)
\psline[linewidth=0.02cm](3.2609375,0.19890624)(4.7609377,0.19890624)
\usefont{T1}{ptm}{m}{n}
\rput(1.8923438,-0.05109375){$v_{1}$}
\usefont{T1}{ptm}{m}{n}
\rput(2.6523438,-0.09109375){$v_{2}$}
\usefont{T1}{ptm}{m}{n}
\rput(3.5523438,-0.09109375){$v_{3}$}
\usefont{T1}{ptm}{m}{n}
\rput(4.2923436,-0.11109375){$v_{4}$}
\usefont{T1}{ptm}{m}{n}
\rput(4.952344,-0.01109375){$v_{5}$}
\usefont{T1}{ptm}{m}{n}
\rput(0.65234375,0.04890625){$P_{5}:$}
\end{pspicture}
}
\begin{center}
    Figure 1.
\end{center}
\end{center}

\noindent\textbf{Example 1.} Let $G$ be the 5-vertex path $P_{5}$, with vertices $v_{1}, v_{2}, v_{3}, v_{4}, v_{5}$ and let its minimum connected dominating set be $D_{c} = \{v_{2}, v_{3}, v_{4}\}$. Then

$$A_{D_{c}}(G) = \left(
  \begin{array}{ccccc}
    0 & 1 & 0 & 0 & 0 \\
    1 & 1 & 1 & 0 & 0 \\
    0 & 1 & 1 & 1 & 0 \\
    0 & 0 & 1 & 1 & 1 \\
    0 & 0 & 0 & 1 & 0 \\
  \end{array}
\right)
$$
The characteristic polynomial of $A_{D_{c}}(G)$ is $\lambda^{5}-3\lambda^{4}- \lambda^{3} +5\lambda^{2} + \lambda -1 = 0$. The minimum connected dominating eigenvalues are $ \lambda_{1} = 2.618$, $\lambda_{2} = 1.618$, $\lambda_{3} = 0.382$, $\lambda_{4} = -1.000$ and $\lambda_{5} = -0.618$. \\
Therefore, the minimum connected dominating energy is $E_{D_{c}}(P_{5}) = 6.236$.
\vspace{2mm}

\noindent\textbf{Example 2.} Consider the following graph
\begin{center}
    % Generated with LaTeXDraw 2.0.2
% Sun Sep 28 22:31:33 IST 2014
% \usepackage[usenames,dvipsnames]{pstricks}
% \usepackage{epsfig}
% \usepackage{pst-grad} % For gradients
% \usepackage{pst-plot} % For axes
%\scalebox{1} % Change this value to rescale the drawing.
{
\begin{pspicture}(0,-1.508125)(6.4428124,1.508125)
\psline[linewidth=0.02cm](1.3609375,0.1696875)(6.1409373,0.1696875)
\psdots[dotsize=0.2](1.3809375,0.1896875)
\psdots[dotsize=0.2](2.2609375,0.1696875)
\psdots[dotsize=0.2](3.0209374,0.1896875)
\psdots[dotsize=0.2](3.7609375,0.1896875)
\psdots[dotsize=0.2](4.4409375,0.1896875)
\psdots[dotsize=0.2](5.2009373,0.1896875)
\psdots[dotsize=0.2](6.1009374,0.1696875)
\psline[linewidth=0.02cm](1.3609375,0.2696875)(1.3609375,0.9696875)
\psline[linewidth=0.02cm](4.4209375,0.2096875)(4.4209375,-0.9303125)
\psdots[dotsize=0.2](4.4009376,-0.9503125)
\psdots[dotsize=0.2](1.3609375,0.9696875)
\psline[linewidth=0.02cm](1.3409375,0.0896875)(1.3409375,-0.5703125)
\psdots[dotsize=0.2](1.3209375,-0.6303125)
\usefont{T1}{ptm}{m}{n}
\rput(1.5223438,1.3196875){$a$}
\usefont{T1}{ptm}{m}{n}
\rput(1.0523437,0.1596875){$b$}
\usefont{T1}{ptm}{m}{n}
\rput(1.2423438,-0.8603125){$c$}
\usefont{T1}{ptm}{m}{n}
\rput(2.2523437,-0.0803125){$d$}
\usefont{T1}{ptm}{m}{n}
\rput(3.0023437,-0.1203125){$e$}
\usefont{T1}{ptm}{m}{n}
\rput(3.7423437,-0.1403125){$f$}
\usefont{T1}{ptm}{m}{n}
\rput(4.4023438,0.5196875){$g$}
\usefont{T1}{ptm}{m}{n}
\rput(5.1723437,-0.1003125){$h$}
\usefont{T1}{ptm}{m}{n}
\rput(6.182344,-0.1203125){$i$}
\usefont{T1}{ptm}{m}{n}
\rput(4.412344,-1.2803125){$j$}
\usefont{T1}{ptm}{m}{n}
\rput(0.30234376,0.1996875){$T:$}
\end{pspicture}
}
\begin{center}
    Figure 2.
\end{center}
\end{center}
Let $G$ be a tree $T$ as shown  above  and let its minimum connected dominating set be $D_{c} =\{b, d, e, f, g, h \}$. Then

$$ A_{D_{c}}(G) =
\left(
  \begin{array}{cccccccccc}
    0 &1& 0& 0& 0& 0& 0& 0& 0& 0 \\
    1 &1& 1& 1& 0& 0& 0& 0& 0& 0 \\
     0 &1& 0& 0& 0& 0& 0& 0& 0& 0\\
    0 &1& 0& 1& 1& 0& 0& 0& 0& 0 \\
    0 &0& 0& 1& 1& 1& 0& 0& 0& 0\\
     0 &0& 0& 0& 1& 1& 1& 0& 0& 0 \\
    0 &0& 0& 0& 0& 1& 1& 1& 0& 1 \\
     0 &0& 0& 0& 0& 0& 1& 1& 1& 0 \\
     0 &0& 0& 0& 0& 0& 0& 1& 0& 0 \\
     0 &0& 0& 0& 0& 0& 1& 0& 0& 0 \\
  \end{array}
\right)
$$
By direct calculation, we get the minimum connected dominating eigenvalues are $ \lambda_{1} = 2.945$, $\lambda_{2} = 2.596$, $\lambda_{3} = 1.896$, $\lambda_{4} = 1.183$, $\lambda_{5} = -1.263$, $ \lambda_{6} = -1.152$, $\lambda_{7} = 0.579$, $\lambda_{8} = 0.000$, $\lambda_{9} = -0.268$ and $\lambda_{10} = -0.516$ . \\
Therefore, the minimum connected dominating energy is $E_{D_{c}}(T) = 12.398$.
\vspace{2mm}

\noindent \textbf{Example 3.} The c-dominating energy of the following graphs can be calculated easily:
\begin{enumerate}
\item $E_{D_{c}}(K_{n}) = (n-2) + \sqrt{n(n-2) + 5}$, where $K_{n}$ is the complete graph of order $n$.
\item $E_{D_{c}}(K_{1, n-1}) = \sqrt{4n-3}$ where $K_{1, n-1}$ is the star graph.
\item $E_{D_{c}}(K_{n \times 2}) = (2n-3) + \sqrt{4n(n-1)-9}$, where $K_{n \times 2}$ is the coctail party graph.
\end{enumerate}

\indent In this paper, we are interested in studying the mathematical aspects of the c-dominating energy of a graph. This paper has organized as follows: The section 1, contains the basic definitions and background of the current topic. In section 2, we show the chemical applicability of c-dominating energy for molecular graphs $G$. The section 3, contains the mathematical properties of c-dominating energy. In the last section, we have characterized, trees, unicyclic graphs and cubic graphs and block graphs with equal minimum dominating energy and c-dominating energy. Finally, we conclude this paper by posing an open problem. 

\section{Chemical Applicability of $E_{D_{c}}(G)$}
We have used the c-dominating energy  for modeling eight representative physical properties like boiling points(bp), molar volumes(mv) at $20^{\circ}C$, molar refractions(mr) at $20^{\circ}C$, heats of vaporization (hv) at $25^{\circ}C$, critical temperatures(ct), critical pressure(cp) and surface tension (st) at $20^{\circ} C$ of the 74 alkanes from ethane to nonanes. Values for these properties were taken from http://www.moleculardescriptors.eu/dataset.htm.  The c-dominating energy $E_{D_{c}}(G)$ was correlated with each of these properties and surprisingly, we can see that the $E_{D_{c}}$ has a good correlation with the heats of vaporization of alkanes with correlation coefficient $r= 0.995$. \\
The following structure-property relationship model has been developed for the c-dominating energy $E_{D_{c}}(G)$.
\begin{eqnarray}
hv = 10 E_{D_{c}}(G) \pm 5.
\end{eqnarray}
\begin{center}
  \includegraphics[width=8cm]{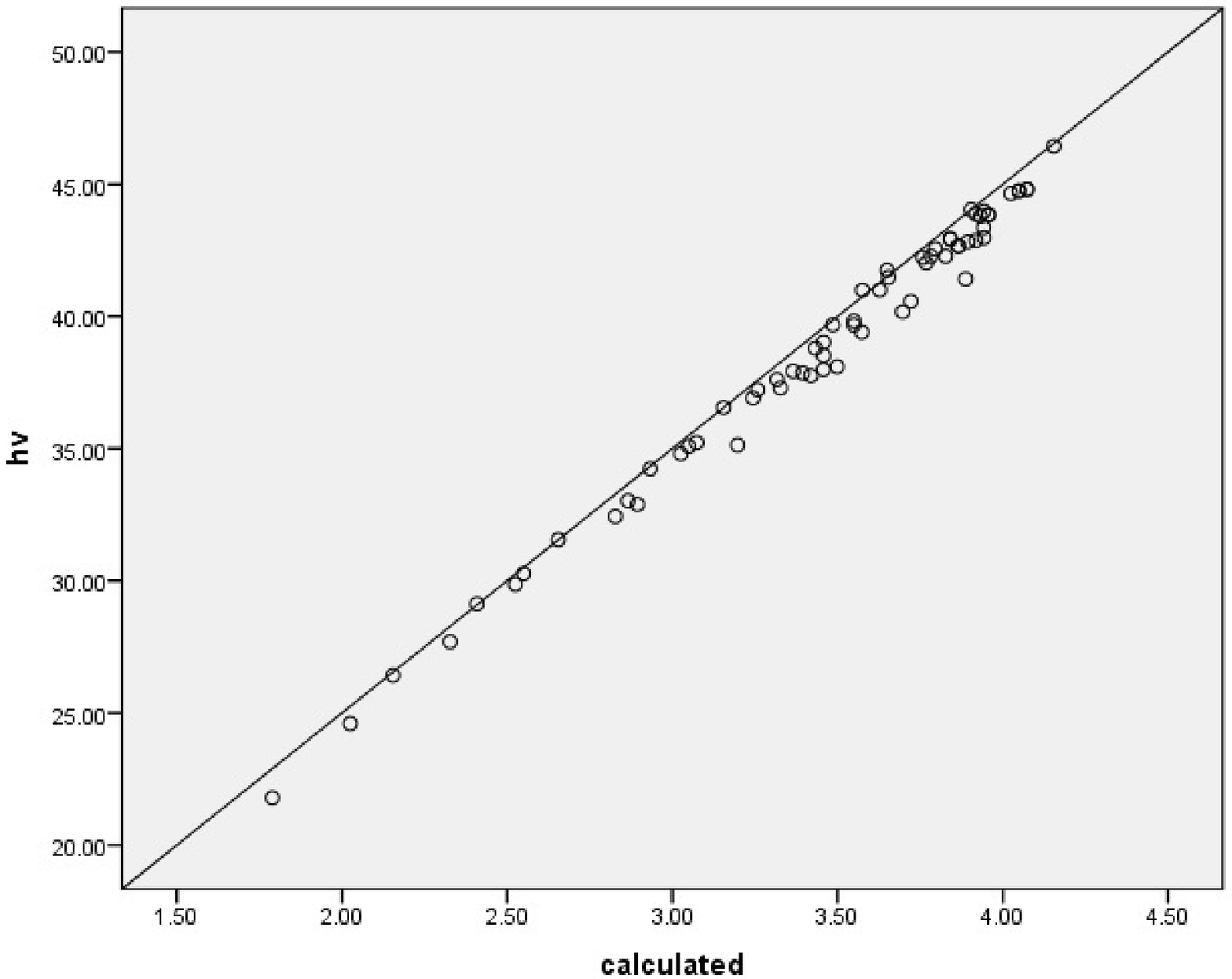}\\
  Figure 3: Correlation of $E_{D_{c}}(G)$ with heats of vaporization of alkanes.
  \end{center}

\section{Mathematical Properties of c-Dominating Energy of Graph}
We begin with the following straightforward observations.
\begin{obs}
Note that the trace of $A_{D_c}(G) = \gamma_c(G)$.
\end{obs}

\begin{obs}
Let $G = (V, E)$ be a graph with $\gamma_c$-set $D_c$. Let $f_n(G,\lambda) = c_0 \lambda^n+c_1 \lambda^{n-1}+\cdots+c_n$ be the characteristic polynomial of $G$. Then
\begin{enumerate}
  \item $c_0 = 1$,
  \item $c_1 = - |D_c| = -\gamma_c(G)$.
\end{enumerate}
\end{obs}

\begin{thm}
If $\lambda_1,\lambda_2,\cdots,\lambda_n$ are the eigenvalues of $A_{D_c}(G)$, then
\begin{enumerate}
  \item $\sum\limits_{i=1}^n \lambda_i = \gamma_c(G)$
  \item $\sum\limits_{i=1}^n \lambda^2_i = 2m+\gamma_c(G)$.
\end{enumerate}
\end{thm}
\begin{proof}
\indent \begin{enumerate}
  \item Follows from Observation 1.
  \item The sum of squares of the eigenvalues of $A_{D_c}(G)$ is just the trace of $A_{D_c}(G)^2$. Therefore
\begin{eqnarray*}
% \nonumber to remove numbering (before each equation)
   \sum\limits_{i=1}^n \lambda^2_i &=& \sum\limits_{i=1}^n \sum\limits_{j =1}^n a_{ij}a_{ji}
	\\
&=& 2\sum\limits_{i < j}(a_{ij})^2+\sum\limits_{i =1}^n(a_{ii})^2
\\
&=& 2m + \gamma_c(G).
\end{eqnarray*}
\end{enumerate}
\end{proof}

We now obtain bounds for $E_{D_c}(G)$ of $G$, similar to McClelland's inequalities \cite{mcclelland} for graph energy.

\begin{thm}
Let $G$ be a graph of order $n$ and size $m$ with $\gamma_{c}(G) = k$. Then

\begin{eqnarray}
E_{D_c}(G) & \le & \sqrt{n(2m+k)}.
\end{eqnarray}
\end{thm}
\begin{proof}
Let $\lambda_1 \ge \lambda_2 \ge \cdots \ge \lambda_n$ be the eigenvalues of $A_{D_c}(G)$. Bearing in mind the Cauchy-Schwarz inequality,
\begin{center}
    $\bigg(\sum\limits_{i=1}^n a_i b_i \bigg)^2 \le \bigg(\sum\limits_{i=1}^n a_i \bigg)^2 \bigg(\sum\limits_{i=1}^n b_i \bigg)^2$
\end{center}
we choose $a_i =1$ and $b_i = |\lambda_i|$, which by Theorem 3 implies
\begin{eqnarray*}
% \nonumber to remove numbering (before each equation)
E_{D_c}^2  &=& \bigg(\sum\limits_{i=1}^n|\lambda_i| \bigg)^2
\\
& \leq & n\bigg(\sum\limits_{i=1}^n|\lambda_i|^2 \bigg)
\\
&=& n \sum\limits_{i=1}^n \lambda^2_i
\\
&=& 2(2m + k).
\end{eqnarray*}
\end{proof}

\begin{thm}
Let $G$ be a graph of order $n$ and size $m$ with $\gamma_c(G) = k$. Let $\lambda_1 \ge \lambda_2 \ge \cdots \ge \lambda_n$ be a non-increasing arrangement of eigenvalues of $A_{D_c}(G)$. Then

\begin{eqnarray}
E_{D_c}(G) & \ge & \sqrt{2mn+nk-\alpha(n)(|\lambda_1|-|\lambda_n|)^2}
\end{eqnarray}

where $\alpha(n) = n [\frac{n}{2}](1-\frac{1}{n}[\frac{n}{2}])$, where $[x]$ denotes the integer part of a real number $k$.
\end{thm}
\begin{proof}
Let $a_1,a_2,\cdots,a_n$ and $b_1,b_2,\cdots,b_n$ be real numbers for which there exist real constants $a,b,A$ and $B$, so that for each $i$, $i = 1,2,\cdots,n,a \le a_{i} \le A$ and $b \le b_i \le B$. Then the following inequality is valid (see \cite{biernacki}).
\begin{eqnarray}
% \nonumber to remove numbering (before each equation)
\mid n\sum\limits_{i =1}^n a_ib_i-\sum\limits_{i =1}^n a_i \sum\limits_{i=1}^n b_i \mid & \le & \alpha(n)(A-a)(B-b)
\end{eqnarray}

where $\alpha(n) = n [\frac{n}{2}](1-\frac{1}{n}[\frac{n}{2}])$. Equality holds if and only if $a_1 = a_2 = \cdots = a_n$ and $b_1 = b_2 = \cdots = b_n$.
\\
We choose $a_i := |\lambda_i|, b_i := |\lambda_i|$, $a = b := |\lambda_n|$ and $A = B := |\lambda_1|, i = 1,2,\cdots,n$, inequality (4) becomes
\begin{eqnarray}
|n \sum\limits_{i=1}^{n}|\lambda_{i}|^2-\bigg(\sum\limits_{i=1}^n|\lambda_i| \bigg)^2| & \le & \alpha(n)(|\lambda_1|-|\lambda_n|)^2
\end{eqnarray}
Since $E_{G_c}(G) = \sum\limits_{i=1}^n|\lambda_i|$, $\sum\limits_{i=1}^n |\lambda_i|^2 = \sum\limits_{i=1}^n |\lambda_{i}|^2  = 2m+k$ and $E_{D_c}(G) \le \sqrt{n(2m+k)}$, the inequality (5) becomes
\begin{eqnarray*}
n(2m+k)-(E_{D_c})^2 & \le & \alpha(n)(|\lambda_1|-|\lambda_n|)^2
\\
(E_{D_c})^2 & \ge & 2mn+nk-\alpha(n)(|\lambda_1|-|\lambda_n|)^2.
\end{eqnarray*}
Hence equality holds if and only if $\lambda_1 = \lambda_2 = \cdots = \lambda_n$.
\end{proof}

\begin{cor}
Let $G$ be a graph of order $n$ and size $m$ with $\gamma_c(G) = k$. Let $\lambda_1 \ge \lambda_2 \ge \cdots \ge \lambda_n$ be a non-increasing arrangement of eigenvalues of $A_{D_c}(G)$. Then
\begin{eqnarray}
E_{D_c}(G) & \geq & \sqrt{2mn+nk-\frac{n^2}{4}(|\lambda_1|-|\lambda_n|)^2}
\end{eqnarray}
\end{cor}
\begin{proof}
Since $\alpha(n) = n [\frac{n}{2}](1-\frac{1}{n}[\frac{n}{2}]) \le \frac{n^2}{4}$, therefore by (3), result follows.
\end{proof}

\begin{thm}
Let $G$ be a graph of order $n$ and size $m$ with $\gamma_{c}(G) = k$. Let $\lambda_1 \ge \lambda_2 \ge \cdots \ge \lambda_n$ be a non-increasing arrangement of eigenvalues of $A_{D_c}(G)$. Then
\begin{eqnarray}
E_{G_c}(G) & \geq & \frac{|\lambda_1||\lambda_2|n+2m+k}{|\lambda_1|+|\lambda_n|}
\end{eqnarray}
\end{thm}
\begin{proof}
Let $a_1,a_2,\cdots,a_n$ and $b_1,b_2,\cdots,b_n$ be real numbers for which there exist real constants $r$ and $R$ so that for each $i$, $i = 1,2,\cdots,n$  holds $ra_i \le b_i \le Ra_i$. Then the following inequality is valid (see \cite{diaz}).
\begin{eqnarray}
\sum\limits_{i=1}^n b^2_i + rR\sum\limits_{i=1}^n a^2_i & \le & (r+R)\sum\limits_{i=1}^n a_i b_i.
\end{eqnarray}
Equality of (8) holds if and only if, for at least one $i$, $1 \le i \le n$ holds $ra_i = b_i = Ra_i$.
\\
For $b_i := |\lambda_i|$, $a_i := 1$ $r := |\lambda_n|$ and $R := |\lambda_1|$, $i = 1,2,\cdots,n$ inequality (8) becomes
\begin{eqnarray}
\sum\limits_{i=n}^{n}|\lambda_i|^2+|\lambda_1||\lambda_n|\sum\limits_{i=1}^{n}1 \le (|\lambda_1|+|\lambda_n|)\sum\limits_{i=1}^n|\lambda_i|.
\end{eqnarray}
Since $\sum\limits_{i=1}^{n}|\lambda_i|^2 = \sum\limits_{i=1}^{n}\lambda_i^2 = 2m+k$, $\sum\limits_{i=1}^n|\lambda_i| = E_{D_c}(G)$, from inequality (9),
\begin{eqnarray*}
2m+k+|\lambda_1||\lambda_n|n & \le & (\lambda_1+\lambda_n)E_{D_c}(G)
\end{eqnarray*}
Hence the result.
\end{proof}
\begin{thm}
Let $G$ be a graph of order $n$ and size $m$ with $\gamma_{c}(G) = k$. If $\xi = |detA_{D_c}(G)|$, then
\begin{eqnarray}
E_{D_c}(G) & \ge & \sqrt{2m+k+n(n-1)\xi^\frac{2}{n}}.
\end{eqnarray}
\end{thm}
\begin{proof}
\begin{eqnarray*}
(E_{D_c}(G))^2 &=& \bigg(\sum\limits_{i=1}^n |\lambda_i| \bigg)^2
\\
&=& \sum\limits_{i=1}^n |\lambda_i|^2+\sum\limits_{i \ne j} |\lambda_i||\lambda_j|.
\end{eqnarray*}
Employing the inequality between the arithmetic and geometric means, we obtain
\begin{eqnarray*}
\frac{1}{n(n-1)} \sum\limits_{i \ne j} |\lambda_i||\lambda_j| & \ge & \bigg(\prod\limits_{i \ne j} |\lambda_i||\lambda_j| \bigg)^{\frac{1}{n(n-1)}}.
\end{eqnarray*}
Thus,
\begin{eqnarray*}
% \nonumber to remove numbering (before each equation)
(E_{D_G})^2 & \ge & \sum\limits_{i=1}^n |\lambda_i|^{2}+n(n-1)\bigg(\prod\limits_{i \ne j} |\lambda_i||\lambda_j| \bigg)^{\frac{1}{n(n-1)}}
\\
& \ge & \sum\limits_{i=1}^n|\lambda_i|^{2}+n(n-1)\bigg(\prod\limits_{i \ne j} |\lambda_i|^{2(n-1)} \bigg)^{\frac{1}{n(n-1)}}
\\
&=& 2m+k+n(n-1)\xi^{\frac{2}{n}}
\end{eqnarray*}
\end{proof}

\begin{lem}
If $\lambda_1(G)$ is the largest minimum connected dominating eigenvalue of $A_{D_c}(G)$, then $\lambda_1 \ge \frac{2m+\gamma_c(G)}{n}$.
\end{lem}
\begin{proof}
Let $X$ be any non-zero vector. Then we have $\lambda_1(A) = max_{_{X \ne 0}} \{\frac{X' A X}{X' X}\}$, see \cite{harary}.
Therefore, $\lambda_1(A_{D_c}(G)) \ge \frac{J'AJ}{J'J} = \frac{2m+\gamma_c(G)}{n}$.
\end{proof}
Next, we obtain Koolen and Moulton's \cite{koolen} type inequality for $E_{D_c}(G)$.

\begin{thm}
If $G$ is a graph of order $n$ and size $m$ and $2m+\gamma_C(G) \ge n$, then
\begin{eqnarray}
% \nonumber to remove numbering (before each equation)
  E_{D_c}(G) & \le & \frac{2m+\gamma_c(G)}{n}+\sqrt{(n-1)\bigg[(2m+\gamma_c(G))-\bigg(\frac{2m+\gamma_c(G)}{n}\bigg)^{2} \bigg]}
\end{eqnarray}
\end{thm}
\begin{proof}
Bearing in mind the Cauchy-Schwarz inequality,
\begin{center}
$\bigg(\sum\limits_{i=1}^n a_i b_i \bigg)^2 \le \bigg(\sum\limits_{i=1}^{n}a_i \bigg)^2 \bigg(\sum\limits_{i=1}^n b_i \bigg)^2$.
\end{center}
Put $a_i =1$ and $b_i = |\lambda_i|$ then
\begin{eqnarray*}
% \nonumber to remove numbering (before each equation)
\bigg(\sum\limits_{i=2}^n a_i b_i \bigg)^2 & \le & (n-1)\bigg(\sum\limits_{i=2}^n b_i \bigg)^2
\\
(E_{D_c}(G)-\lambda_1)^2 & \le &(n-1)(2m+\gamma_c(G)-\lambda^2_1)
\\
E_{D_c}(G) & \le & \lambda_1+\sqrt{(n-1)(2m+\gamma_c(G)-\lambda^2_1)}.
\end{eqnarray*}
Let
\begin{eqnarray}
f(x) &=& x+\sqrt{(n-1)(2m+\gamma_c(G)-x^2)}
\end{eqnarray}
For decreasing function
\begin{eqnarray*}
f'(x) & \le & 0
\\
\Rightarrow 1-\frac{x(n-1)}{\sqrt{(n-1)(2m+\gamma_c(G)-x^2)}} & \le & 0
\\
x & \ge & \sqrt{\frac{2m+\gamma_c(G)}{n}}.
\end{eqnarray*}
Since $(2m+k) \ge n$, we have $\sqrt{\frac{2m+\gamma_{c}(G)}{n}} \le \frac{2m+\gamma_c(G)}{n} \le \lambda_1$.
Also $f(\lambda_1) \le f\bigg(\frac{2m+\gamma_c(G)}{n} \bigg)$.
\begin{center}
i.e $E_{D_c}(G) \le f(\lambda_1) \le f\bigg( \frac{2m+\gamma_c(G)}{n} \bigg)$.
\\
i.e $E_{D_c}(G) \le  f\bigg( \frac{2m+\gamma_c(G)}{n} \bigg)$
\end{center}
Hence by (12), the result follows.
\end{proof}

\section{Graphs with equal Dominating and c-Dominating Energy}
Its a natural question to ask that for which graphs the dominating energy and c-dominating energy are equal. To answer this question, we characterize graphs with equal dominating energy and c-dominating energy. The graphs considered in this section are trees, cubic graphs, unicyclic graphs, block graphs and cactus graphs.

\begin{thm}
Let $G= T$ be a tree with at least three vertices, then $E_{D}(G) = E_{D_c}(G)$ if and only if every internal vertex of $T$ is a support vertex.
\end{thm}
\begin{proof}
Let $G = T$ be a tree of order at least 3. Let $F = \{u_{1}, u_{2}, \cdots, u_{k}\}$ be the set of internal vertices of $T$.
 Then clearly $F$ is the minimal dominating set of $G$. Therefore in $A_{D}(G)$ the values of $u_{i} = 1$ in the diagonal entries. Further, observe that $\langle F\rangle$ is connected. Hence $F$ is the minimal connected dominating set. Therefore, $A_{D}(G) = A_{D_{c}}(G)$. In general, $A_{D}(G) = A_{D_{c}}(G)$ is true if  every minimum dominating set is connected. In other words, $A_{D}(G) = A_{D_{c}}(G)$ if  $\gamma(G) = \gamma_{c}(G)$. Therefore, the result follows from Theorem 2.1 in \cite{arumugam}.
\end{proof}

In the next three theorems we characterize unicyclic graphs with $A_{D}(G) = A_{D_{c}}(G)$. Since, $A_{D}(G) = A_{D_{c}}(G)$ if  $\gamma(G) = \gamma_{c}(G)$. Therefore, the proof of our next three results follows from Theorem 2.2, Theorem 2.4 and Theorem 2.5 in \cite{arumugam}.
\begin{thm}
Let $G$ be a unicyclic graph with cycle $C= u_{1}u_{2} \cdots, u_{n}u_{1}$ $n\geq 5$ and let $X = \{v \in C : d_{G}(v) \geq 2 \}$. Then $E_{D}(G) = E_{D_c}(G)$ if  the following conditions hold:
\begin{enumerate}
  \item (a). Every $v \in V- N[X]$ with $d_{G}(V) \geq 2$ is a support vertex.
  \item (b). $\langle X \rangle$ is connected and $|X| \leq 3$.
  \item (c). If $\langle X \rangle = P_{1} \text{or} P_{3}$, both vertices in $N(X)$ of degree at least 3 are supports and if $\langle X \rangle = P_{2}$, at least one vertex in $N(X)$ of degree at least three is a support.
\end{enumerate}
\end{thm}

\begin{thm}
Let $G$ be unicyclic graph with $|V(G)| \geq 4$ containing a cycle $C= C_{3}$, and let $X = \{v \in C : d_{G}(v)= 2 \}$. Then $E_{D}(G) = E_{D_c}(G)$  if the following conditions hold:
\begin{enumerate}
  \item (a). Every $v \in V- N[X]$ with $d_{G}(V) \geq 2$ is a support vertex.
  \item (b). There exists some unique $v \in C$ with $d_{G}(v) \geq 3$ or for every $v \in C$ of $d_{G}(v) \geq 3$ is a support.
  \end{enumerate}
\end{thm}

\begin{thm}
Let $G$ be unicyclic graph with $|V(G)| \geq 5$ containing a cycle $C= C_{4}$, and let $X = \{v \in C : d_{G}(v)= 2 \}$. Then $E_{D}(G) = E_{D_c}(G)$  if the following conditions hold:
\begin{enumerate}
  \item (a). Every $v \in V- N[X]$ with $d_{G}(V) \geq 2$ is a support vertex.
  \item (b). If $|X|=1$, all the three remaining vertices of $C$ are supports and if $|X| \geq 2$, $C$ contains at least one support.
  \end{enumerate}
\end{thm}

\begin{thm}
Let $G$ be a connected cubic graph of order $n$, Then  $E_{D}(G) = E_{D_c}(G)$ if  $G \cong K_{4}, \overline{C_{6}, K_{3,3,}}$, $G_{1}$ or $G_{2}$ where $G_{1}$ and $G_{2}$ are given in Fig. 4.
\end{thm}

\begin{center}
    % Generated with LaTeXDraw 2.0.2
% Fri Oct 26 13:08:52 IST 2018
% \usepackage[usenames,dvipsnames]{pstricks}
% \usepackage{epsfig}
% \usepackage{pst-grad} % For gradients
% \usepackage{pst-plot} % For axes
\scalebox{1} % Change this value to rescale the drawing.
{
\begin{pspicture}(0,-2.8389063)(10.88,2.7989063)
\psellipse[linewidth=0.04,dimen=outer](2.41,0.5989063)(2.37,2.2)
\psdots[dotsize=0.2](1.04,2.3789062)
\psdots[dotsize=0.2](3.64,2.4189062)
\psdots[dotsize=0.2](0.1,1.0189062)
\psdots[dotsize=0.2](0.26,-0.34109375)
\psdots[dotsize=0.2](1.22,-1.2610937)
\psdots[dotsize=0.2](3.48,-1.2810937)
\psdots[dotsize=0.2](4.5,-0.34109375)
\psdots[dotsize=0.2](4.7,1.0789063)
\psellipse[linewidth=0.04,dimen=outer](8.47,0.57890624)(2.37,2.2)
\psdots[dotsize=0.2](7.1,2.3589063)
\psdots[dotsize=0.2](9.7,2.3989062)
\psdots[dotsize=0.2](6.16,0.99890625)
\psdots[dotsize=0.2](6.32,-0.36109376)
\psdots[dotsize=0.2](7.28,-1.2810937)
\psdots[dotsize=0.2](9.54,-1.3010937)
\psdots[dotsize=0.2](10.56,-0.36109376)
\psdots[dotsize=0.2](10.76,1.0589062)
\psline[linewidth=0.04cm](1.0,2.3589063)(3.48,-1.2610937)
\psline[linewidth=0.04cm](1.22,-1.2010938)(3.64,2.4589062)
\psline[linewidth=0.04cm](0.12,1.0189062)(4.7,1.0389062)
\psline[linewidth=0.04cm](0.3,-0.30109376)(4.54,-0.34109375)
\psline[linewidth=0.04cm](7.08,2.3389063)(9.52,-1.2810937)
\psline[linewidth=0.04cm](6.3,-0.36109376)(9.7,2.4589062)
\psline[linewidth=0.04cm](6.12,0.9789063)(10.78,1.0189062)
\psline[linewidth=0.04cm](7.26,-1.2810937)(10.6,-0.34109375)
\usefont{T1}{ptm}{m}{n}
\rput(2.1714063,-2.0710938){$G_{1}$}
\usefont{T1}{ptm}{m}{n}
\rput(8.471406,-2.0710938){$G_{2}$}
\usefont{T1}{ptm}{m}{n}
\rput(5.380781,-2.6110938){Figure 4}
\end{pspicture}
}
\end{center}

\begin{thm}
Let $G$ be a block graph of with $l \geq 2$. Then $E_{D}(G) = E_{D_c}(G)$ if  every cutvertex of $G$ is an end block cutvertex.
\end{thm}
\begin{proof}
Since $E_{D}(G) = E_{D_c}(G)$ if $\gamma(G) = \gamma_{c}(G)$. Therefore, the result follows from Theorem 2 in \cite{chen}.
\end{proof}

We conclude this paper by posing the following open problem for the researchers:\\
\textbf{Open Problem:} Construct non- cospectral graphs with unequal domination and connected domination numbers having equal dominating energy and c-dominating energy.

\end{document}